\documentclass[preprint,12pt,1p]{elsarticle}
\makeatother
\usepackage{graphics}
\usepackage{epsfig}
\usepackage{mathptmx}
\usepackage{amsmath}
\usepackage{amssymb}
\usepackage{hyperref}
\usepackage{fullpage}
\usepackage{amsthm}
\usepackage{lineno}

\theoremstyle{plain}
\newtheorem{theorem}{Theorem}[section]

\newtheorem{lemma}[theorem]{Lemma}

\theoremstyle{definition}
\newtheorem{definition}{Definition}[section]
\theoremstyle{example}

\theoremstyle{remark}

\newlength{\defbaselineskip}
\makeatletter
\def\ps@pprintTitle{
  \let\@oddhead\@empty
  \let\@evenhead\@empty
  \let\@oddfoot\@empty
  \let\@evenfoot\@oddfoot
}
\journal{Computer aided geometric design}
\begin{document}
\begin{frontmatter}
\title{Hemi-slant submanifolds of nearly Kaehler manifolds}
\author[n1]{Mehraj Ahmad Lone\corref{cor1}}
\ead{mehraj.jmi@gmail.com}
\author[n3]{Mohammad Hasan Shahid}
\address[n1]{Department of Mathematics, Central University of Jammu, Jammu, 180011, India.}
\address[n3]{Department of Mathematics, Jamia Millia Islamia, New Delhi-110 025, India}
\cortext[cor1]{Corresponding author}
\begin{abstract}
In the present paper, we study the hemi-slant submanifolds of nearly Kaehler manifolds. We study the integrability of distributions involved in the definition of hemi-slant submanifolds. some results are worked out on totally umbilical hemi-slant submanifolds. we study the cohomology class for hemi-slant submanifolds of nearly Kaehler manifolds.

\end{abstract}
\begin{keyword}
\texttt Cohomology class,  Hemi-slant submanifolds, Nearly Kaehler manifold, Umbilical submanifold.

2010 Mathematics Subject Classification:  53B25, 53C40, 53C55, 57R19.
\end{keyword}
\end{frontmatter}
\section{Introduction} The study of slant submanifolds was started in 1990 by Chen \cite{biha8} which is the generalization of CR-submanifolds introduced by Bejancu \cite{biha1}. The theory of slant submanifolds become very rich area of research for geometers. Slant submanifolds were studied by diffrent geometers in different ambient spaces. In 1994 Papaghiuc   generalizes the slant submanifolds as semi-slant submanifolds \cite{biha11}. Carriazo introduced the idea of bi-slant submanifolds as generalization of semi-slant submanifolds \cite{biha3}. The hemi-slant submanifolds are one of the classes of bi-slant submanifolds. The hemi-slant submanifolds were first named as anti-slant submanifolds. In \cite{biha5} B. Sahin named these submanifolds as hemi-slant submanifolds, as anti-slant seems to refer no slant factor. since then these submanifolds were studied in different kinds of structures.

In this paper, we first collect some results and definitions for ready references in the preliminary section. In third section we study the hemi-slant submanifolds of nearly Kaehler manifolds. In fourth section we obtain some results on totally umbilical hemi-slant submanifolds and in the last section we study  cohomology class for hemi-slant submanifolds  of nearly Kaehler manifolds.

\section{Preliminaries}

Let $f:\mathcal{N}\rightarrow \overline{\mathcal{N}}$ be an isometrical  immersion of a Riemannian manifolds $\mathcal{N}$ into $\overline{\mathcal{N}}$. The Riemannian metric  for $\mathcal{N}$ and  $\overline{\mathcal{N}}$ is denoted by the same symbol $g$. Let $T\mathcal{N}$ and $T^{\perp} \mathcal{N}$ denote the Lie algebra of vector field and set of all normal vector fields on $\mathcal{N}$ respectively. The Levi-Civita connection on $T\mathcal{N}$ and  $T^{\perp} \mathcal{N}$ is denoted by $\nabla$ and $\overline{\nabla}$. the Gauss and Weingarten equation are give as

\begin{eqnarray}\label{a1}
\overline{\nabla}_{X}Y = \nabla_{X}Y + \Omega(X,Y),
\end{eqnarray}
\begin{eqnarray}\label{a2}
\overline{\nabla}_{X}V = - S_{V}X + \nabla_{X}^{\perp}Y,
\end{eqnarray}

for any $X, Y$ in $T\mathcal{N}$ and $V$ in $T^{\perp} \mathcal{N}$. Where $\Omega$ is the second fundamental form and $S$ is the shape operator. The second fundamental form and the shape operator are related by the following equation
\begin{eqnarray}\label{a3}
g(\Omega(X,Y), V) = g(S_{V}X, Y).
\end{eqnarray}
The mean curvature vector $\mathcal{H}$ on a submanifold $\mathcal{N}$ is given by the following relation
\begin{eqnarray}\label{a4}
\mathcal{H} = \frac{1}{l}\sum_{i=1}^{l}\Omega(e_{i},e_{i}),
\end{eqnarray}
where $l$ denotes the dimension of submanifold $\mathcal{N}$ and $\{e_{1},  ..., e_{l}\}$ is the local orthonormal frame.
 The submanifold $\mathcal{N}$ is said to be a totally umbilical submanifold of a Riemannian manifold $\mathcal{\overline{N}}$, if for the mean curvature vector $\mathcal{H}$ The following relation holds
\begin{eqnarray}\label{a5}
\Omega(X,Y) = g(X,Y)\mathcal{H}.
\end{eqnarray}
Using (\ref{a5}), the gauss and Weingarten equation takes the form
\begin{eqnarray}\label{a5n1}
\overline{\nabla}_{X}Y = \nabla_{X}Y + g(X,Y)\mathcal{H},
\end{eqnarray}
 and
\begin{eqnarray}\label{a5n2}
\overline{\nabla}_{X}V = -Xg(\mathcal{H}, V) + \nabla_{X}^{\perp}V.
\end{eqnarray}
for $X, Y \in T\mathcal{N}$ and $V \in T^{\perp}\mathcal{N}$.

let $\mathcal{P}X$ and $\mathcal{N}X$ denotes the tangential and normal components of $(\overline{\nabla_{X}}J)\mathcal{H}$ then
\begin{eqnarray}\label{a5n3}
\mathcal{P}X = -Xg(\mathcal{H}, \mathcal{JH}) - TX g(\mathcal{H},\mathcal{H}) -t\nabla_{X}^{\perp}\mathcal{H}
\end{eqnarray}
and
\begin{eqnarray}\label{a5n4}
\mathcal{N}X =  \nabla_{X}^{\perp}J\mathcal{H} - NX g(\mathcal{H},\mathcal{H}) - n\nabla_{X}^{\perp}\mathcal{H}
\end{eqnarray}
 for $X \in T\mathcal{N}$.

A submanifold $\mathcal{N}$ is said to be totally geodesic if $\Omega(X,Y) = 0$, for any $X, Y \in T\mathcal{N}$ and if $\mathcal{H} = 0$, the submanifold $\mathcal{N}$ is   minimal.

 Now for a submanifold $\mathcal{N}$ of a Riemannian manifold $\mathcal{\overline{N}}$, we can write $X \in T\mathcal{N}$ as
 \begin{eqnarray}\label{a6}
 JX = TX + NX,
 \end{eqnarray}
where $PX$ and $NX$ are the tangential and normal parts of $JX$ respectively.
In the same way, for any $V \in T^{\perp}\mathcal{N}$, we have
\begin{eqnarray}\label{a7}
JV = tV + nV,
\end{eqnarray}
where $tV$ and $nV$ are tangential and normal components of $JV$.
 The covariant derivative of tangential and normal parts of (\ref{a6}) and (\ref{a7}) is given as

\begin{eqnarray}\label{a8}
(\overline{\nabla}_{X}T)Y = \nabla_{X}TY - T\nabla_{X}Y,
\end{eqnarray}
\begin{eqnarray}\label{a9}
(\overline{\nabla}_{X}N)Y = \nabla_{X}^{\perp}NY - N\nabla_{X}Y,
\end{eqnarray}
\begin{eqnarray}\label{10}
(\overline{\nabla}_{X}t)V = \nabla_{X}tV - t\nabla_{X}^{\perp}V,
\end{eqnarray}
and
\begin{eqnarray}\label{a11}
(\overline{\nabla}_{X}f)V = \nabla_{X}^{\perp}fV - f\nabla_{X}^{\perp}V.
\end{eqnarray}

The manifold $\mathcal{\overline{N}}$  with almost complex structure $J$ and Riemannian metric $g$ is said to be a nearly Kaehler manifold if
\begin{equation}\label{a12}
(\overline{\nabla}_{X}J)Y + (\overline{\nabla}_{Y}J)X = 0,
\end{equation}
for any vector field $X, Y$ in $T\mathcal{N}$. The above condition is equivalent to the following equation
\begin{equation}\label{a13}
(\overline{\nabla}_{X}J)X  = 0,
\end{equation}
 for $X$ in $T\mathcal{N}$. For a nearly Kaehler manifold, we also have the following facts:
 \begin{eqnarray}\label{a14}
 J^{2} = -I \hspace{.5cm}and \hspace{.5cm} g(JX, JY)= g(X,Y),
 \end{eqnarray}
for all vector field $X , Y \in \mathcal{N}.$

We know that the Kaehler manifold is always a nearly Kaehler manifolds but converse is not always true.

 The covariant derivative of the complex structure $J$ is defined as
 \begin{eqnarray}\label{a15}
(\overline{\nabla}_{X}J)Y = \overline{\nabla}_{X}JY - J\overline{\nabla}_{X}Y,
\end{eqnarray}
or
\begin{eqnarray}\label{a16}
 \overline{\nabla}_{X}JX = J\overline{\nabla}_{X}X.
\end{eqnarray}
For a hemi-slant submanifold  $\mathcal{N}$ of a nearly Kaehler manifold $\mathcal{\overline{N}}$. The following facts can easily seen from \cite{biha5}
\begin{eqnarray}\label{a17}
T^{2}X = - cos^{2}\theta X,
\end{eqnarray}
\begin{eqnarray}\label{a18}
g(TX, TY) = cos^{2}\theta g(X,Y),
\end{eqnarray}
and
\begin{eqnarray}\label{a19}
g(NX, NY) = sin^{2}\theta g(X,Y),
\end{eqnarray}

for $X,Y \in D^{\theta}$.

Let $\mathcal{\overline{N}}$ be nearly Kaehler manifold and  $\mathcal{N}$ be a submanifolds of $\mathcal{\overline{N}}$. there are certain important classes of submanifolds viz totally real submanifolds, holomorphic submanifolds and CR-submanifolds as a generalization of totally real and holomorphic submanifolds. The CR-submanifolds were generalized by Chen as slant submanifolds. The hemi-slant submanifolds is a generalization  of slant submanifolds.

\begin{definition}
Let $\mathcal{N}$ be submanifold of nearly Kaehler  manifold $\mathcal{\overline{N}}$, then $\mathcal{N}$ is said to be a hemi-slant submanifold if there exist two orthogonal distributions $D^{\theta}$ and $D^{\perp}$ on $\mathcal{N}$ such that

\noindent(i) $T\mathcal{N}$ = $D^{\theta}\oplus D^{\perp}$

\noindent(ii) $D^{\theta}$ is a slant distribution with slant angle $\theta \ne \frac{\pi}{2}$,

\noindent(iii) $D^{\perp}$ is a totally real, that is $JD^{\perp}\subseteq T^{\perp}\mathcal{N}$,

it is clear from above that CR-submanifolds and slant submanifolds are hemi-slant submanifolds with slant angle $\theta = \frac{\pi}{2}$ and $D^{\theta}$ = {0}, respectively.
\end{definition}

\section{Hemi-slant submanifolds of nearly Kaehler manifolds}
In  this section, we study the integrability  and the minimality of the distributions involved in the definition. We try to prove some lemmas which are very important to prove for next two sections.

\begin{lemma}
Let $\mathcal{N}$ be proper hemi-slant submanifold of nearly Kaehler manifold $\mathcal{\overline{N}}$. then we have
\begin{equation*}\label{u1}
J(D^{\perp}) \perp N(D^{\theta}).
\end{equation*}
\end{lemma}
\begin{proof}
For any $X \in D^{\theta}$ and $Z \in D^{\perp}$, using (\ref{a6})
\begin{equation*}\label{u2}
g(JX, JZ) = (TX + NX, JZ) = g(NX,JZ).
\end{equation*}
We know from (\ref{a14})
\begin{equation*}\label{u3}
g(JX, JZ) = g(X, Z) = 0,
\end{equation*}
which implies that
\begin{equation*}\label{u4}
g(NX,JZ) = 0,
\end{equation*}
which shows that  $N(D^{\theta}) \perp J(D^{\perp})$
\end{proof}
In the light of above lemma the normal bundle $T^{\perp}\mathcal{N}$ can be decomposed as
\begin{equation}\label{u5}
T^{\perp}\mathcal{N} = J(D^{\perp}) \oplus N(D^{\theta})\oplus \mu,
\end{equation}
where $\mu$ is the invariant distribution of $T^{\perp}\mathcal{N}$ with respect to $J$.
\begin{lemma}
Let $\mathcal{N}$ be proper hemi-slant submanifold of nearly Kaehler manifold $\mathcal{\overline{N}}$. then we have
\begin{equation}\label{u6}
TD^{\perp} = \{0\} \hspace{.5cm} and \hspace{.5cm} TD^{\theta} = D^{\theta}.
\end{equation}
\end{lemma}
\begin{proof}
The proof follow from  (\ref{a6}), (\ref{a14}) and (\ref{a17})trivially, so we omit the proof.
\end{proof}

\begin{theorem}
Let $\mathcal{N}$ be a submanifold of nearly Kaehler manifold $\mathcal{\overline{N}}$, then for any $X, Y \in T\mathcal{N}$, we have the following
\begin{eqnarray}\label{u7}
\nabla_{X}TY  - S_{NY}X - T\nabla_{X}Y  - 2t\Omega(X,Y) + \nabla_{Y}TX - S_{NX}Y  - T\nabla_{Y}X = 0,
\end{eqnarray}
and
\begin{eqnarray*}\label{u8}
 \Omega(X,TY) +\nabla_{X}^{\perp}NY - N\nabla_{X}Y - 2n\Omega(X,Y) + \Omega(Y,TX) + \nabla_{Y}^{\perp}NX  - N\nabla_{Y}X = 0.
\end{eqnarray*}
\end{theorem}
\begin{proof}
For a nearly Kaehler manifold, we know that
\begin{eqnarray*}\label{u9}
(\overline{\nabla}_{X}J)Y + (\overline{\nabla}_{Y}J)X = 0,
\end{eqnarray*}
which implies that
\begin{eqnarray*}\label{u10}
\overline{\nabla}_{X}JY - J\overline{\nabla}_{X}Y + \overline{\nabla}_{Y}JX - J\overline{\nabla}_{Y}X = 0.
\end{eqnarray*}
Using (\ref{a1}) and  (\ref{a6}), we have
\begin{eqnarray*}\label{u11}
 \nonumber\overline{\nabla}_{X}TY + \overline{\nabla}_{X}NY - J\nabla_{X}Y - J\Omega(X,Y)
+ \overline{\nabla}_{Y}TX + \overline{\nabla}_{Y}NX - J\nabla_{Y}X - J\Omega(Y,X) = 0.
\end{eqnarray*}
Using (\ref{a1}), (\ref{a2}), (\ref{a6}) and (\ref{a7}), we get
\begin{eqnarray*}\label{u12}
 \nonumber\nabla_{X}TY + \Omega(X,TY) - S_{NY}X +\nabla_{X}^{\perp}NY -T\nabla_{X}Y - N\nabla_{X}Y - t\Omega(X,Y) -n\Omega(X,Y)\\
+\nabla_{Y}TX + \Omega(Y,TX) - S_{NX}Y +\nabla_{Y}^{\perp}NX -T\nabla_{Y}X - N\nabla_{Y}X - t\Omega(Y,X) -n\Omega(Y,X) = 0.
\end{eqnarray*}
Comparing the tangential and normal components, we  have the required results.
\end{proof}
In the following theorem, we prove the integrability of of the slant distribution.
\begin{theorem}
Let $\mathcal{N}$ be hemi-slant submanifold of a nearly Kaehler manifold, then the slant distribution $D^\theta$ is integrable if and only if
\begin{eqnarray}\label{u13}
 \nabla_{X}TY + \nabla_{Y}TX - S_{NX}Y - S_{NY}X - 2T\nabla_{Y}X -2t\Omega(X,Y) \in D^{\theta},
\end{eqnarray}
for $X,Y \in D^{\theta}.$
\end{theorem}
\begin{proof}
We know that the slant distribution is integrable if and only $[X,Y]\in D^{\theta}$ for any $X,Y \in D^{\theta}$. On the other hand, we know from (\ref{u6}), that $TD^{\theta} = D^\theta$, which infer that $[X,Y] \in D^{\theta}$ if and only if $T[X,Y] \in D^{\theta}$. From (\ref{u7}), we have
\begin{eqnarray*}\label{u14}
\nabla_{X}TY  - S_{NY}X - T\nabla_{X}Y  - 2t\Omega(X,Y) + \nabla_{Y}TX - S_{NX}Y  - \nabla_{Y}X = 0,
\end{eqnarray*}
for any $X,Y \in D^{\theta}$.
From which we conclude that
\begin{eqnarray*}\label{u15}
T[X,Y] = \nabla_{X}TY + \nabla_{Y}TX - S_{NX}Y - S_{NY}X - 2T\nabla_{Y}X -2t\Omega(X,Y).
\end{eqnarray*}
Using (\ref{u6}) and from above equation, we obtain that $T[X,Y] \in D^{\theta}$ if and only if (\ref{u13}) satisfies.
\end{proof}

Now we prove the integrability of the anti-invariant distribution
\begin{theorem}
Let $\mathcal{N}$ be a hemi-slant submanifold of nearly Kaehler manifold $\mathcal{\overline{N}}$ the the totally real  distribution $D^{\perp}$ is integrable if and only if
\begin{eqnarray*}\label{u16}
S_{NZ}W + S_{NW}Z - 2T\nabla_{W}Z - 2t\Omega(Z,W) = 0.
\end{eqnarray*}
for any $Z, W \in D^{\perp}$.
\end{theorem}
\begin{proof}
For $Z, W \in D^{\perp}$. Using (\ref{u7}), we have
\begin{eqnarray*}\label{u17}
\nabla_{Z}TW  - S_{NW}Z - T\nabla_{Z}W  - 2t\Omega(Z,W) + \nabla_{W}TZ - S_{NZ}W  - T\nabla_{W}Z = 0.
\end{eqnarray*}
Using (\ref{u6}), from above equation, we have
\begin{eqnarray*}\label{u18}
- S_{NW}Z - T\nabla_{Z}W  - 2t\Omega(Z,W) - S_{NZ}W  - T\nabla_{W}Z = 0,
\end{eqnarray*}
from which we infer that
\begin{eqnarray*}\label{u19}
T[Z,W] = - S_{NZ}W  - S_{NW}Z - 2T\nabla_{W}Z - 2t\Omega(Z,W)= 0
\end{eqnarray*}
Hence from above equation we obtain the result.
\end{proof}

Next, we prove the minimality of  the slant distribution $D^{\theta}$.

%
\begin{lemma}
Let $\mathcal{N}$ be hemi-slant submanifold of nearly Kaehler manifold $\mathcal{\overline{N}}$ then the distribution $D^{\theta}$ is minimal if and only if the normal bundle is parallel and 
\begin{eqnarray*}\label{u24}
g(A_{JW}Z + sec^{2}\theta A_{JW}Z, TX )=0,
\end{eqnarray*}
for $X \in D^{\theta}$ and $W \in D^{\perp}$.
\end{lemma}
\begin{proof}
For $X \in D^{\theta}$ and $W \in D^{\perp}$, we have
\begin{eqnarray*}\label{u25}
g(\nabla_{X}X + \nabla_{sec\theta TX}sec\theta TX,W) =   g(\nabla_{X}X, W) + g(\nabla_{sec\theta TX}sec\theta TX,W)
\end{eqnarray*}
 Using (\ref{a1}) and (\ref{a2}) and the lemma (3.2), we have
\begin{eqnarray*}\label{u26}
\nonumber g(\nabla_{X}X + \nabla_{sec\theta TX}sec\theta TX,W)  &=& g(\overline{\nabla}_{X}JX, JW) + sec^{2}\theta g(\overline{\nabla}_ {X}JX, JW)
\nonumber\\ &=& g(\overline{\nabla}_{X}TX, JW) + g(\overline{\nabla}_{X}NX, JW) \\ & & + sec^{2}\theta g(\overline{\nabla}_ {X}TX, JW)+ sec^{2}\theta g(\overline{\nabla}_ {X}NX, JW)
\nonumber\\ &=& g(A_{JW}X,TX) + g(\nabla^{\perp}_{X}NX, JW) \\ & & + sec^{2}\theta g(A_{JW}X,TX)+ sec^{2}\theta g(\nabla^{\perp}_{X}NX, JW),
\end{eqnarray*}
Now using the parallel condition of normal bundle, we conclude that
\begin{eqnarray*}\label{u27}
g(\nabla_{X}X + \nabla_{sec\theta TX}sec\theta TX,W) = g(A_{JW}Z + sec^{2}\theta A_{JW}Z, TX )
\end{eqnarray*}
frow which we have the result.
\end{proof}

\section{Totally umbilical hemi-slant submanifolds}
In the present section, we deal with the hemi-slant submanifolds of nearly cosymplectic manifold. Throughout this section we treat $\mathcal{N}$ a totally umbilical hemi-slant submanifold of a nearly Kaehler manifolds $\mathcal{\overline{N}}$

\begin{theorem}
Let $\mathcal{N}$ be a totally umbilical hemi-slant submanifolds of a nearly Kaehler manifold  $\mathcal{\overline{N}}$. Then one of the following statements  is necessarily holds
(i) $\mathcal{H} = 0$ or
(ii) $\mathcal{H} \perp JW$ or
(iii) $dim D^{\perp} = 1$ or
(iv) $\mathcal{N}$ is proper hemi-slant submanifold.
\end{theorem}

\begin{proof}
We know that $\mathcal{\overline{N}}$ is nearly Kaehler manifold. For any $Z, W \in D^{\perp}$, we have
\begin{equation*}\label{uu1}
(\overline{\nabla}_{Z}J)W + (\overline{\nabla}_{W}J)Z = 0,
\end{equation*}
Using (\ref{a15}) implies that
\begin{equation*}\label{uu2}
\overline{\nabla}_{Z}JW + J\overline{\nabla}_{Z}W + \overline{\nabla}_{W}JZ + J\overline{\nabla}_{W}Z  = 0.
\end{equation*}
Using the equation ( \ref{a5n1}) and  (\ref{a5n2}), we obtain
\begin{eqnarray*}\label{uu3}
\nonumber -Zg(\mathcal{H},JW) + \nabla^{\perp}_{Z}JW - J\nabla_{Z}W - g(Z,W)J\mathcal{H} - Wg(\mathcal{H},JZ)\\ + \nabla^{\perp}_{W}JZ - J\nabla_{W}Z - g(W,Z)J\mathcal{H} = 0.
\end{eqnarray*}
Taking inner product with $Z$, we get
\begin{eqnarray*}\label{uu4}
 -g(Z,Z)g(\mathcal{H},JW) - g(Z,W)g(Z,J\mathcal{H}) - g(W,Z)g(\mathcal{H},JZ) - g(W,Z)g(Z,J\mathcal{H}) = 0,
\end{eqnarray*}

from which we conclude that
\begin{eqnarray}\label{uu5}
 g(Z,W)g(Z,J\mathcal{H}) = \|Z\|^{2}g(W,J\mathcal{H}).
\end{eqnarray}
Interchange the role of $Z$ and $W$, we have
\begin{eqnarray}\label{uu6}
 g(W,Z)g(W,J\mathcal{H}) = \|W\|^{2}g(Z,J\mathcal{H}).
\end{eqnarray}
From (\ref{uu5}) and (\ref{uu6}), we get
\begin{eqnarray}\label{uu7}
 g(W,J\mathcal{H}) = \frac{g(Z,W)^{2}}{\|Z\|^{2}\|W\|^{2}}g(W,J\mathcal{H}).
\end{eqnarray}
The  solution of (\ref{uu7}) hold with respect to the following cases:\\
(i) $\mathcal{H} = 0$, then we say that $\mathcal{N}$ is totally geodesic submanifold,\\
(ii) $\mathcal{H} \perp JW$, which implies that $\mathcal{H} \in \mu$,\\
(iii) $X$ is parallel to $W$, then $dim{D^{\perp}} = 1,$ \\
(iv) $\mathcal{N}$ is proper hemi-slant submanifold as $\mathcal{H} \in \mu$ implies that $D^{\theta} \neq 0.$
\end{proof}
\begin{theorem}
Let $\mathcal{N}$ be a totally umbilical proper hemi-slant submanifolds of a nearly Kaehler manifold  $\mathcal{\overline{N}}$ such that $\mathcal{H} \in \mu$ and $\nabla_{X}^{\perp}\mathcal{H} \in \mu $ , Then$\mathcal{N}$ is totally geodesic in $\mathcal{\overline{N}}$,
\end{theorem}
\begin{proof}
Since $\mathcal{N}$ is a nearly Kaehler manifold. For $X \in T\mathcal{N}$, we have
\begin{eqnarray*}
\overline{\nabla}_{X}JX = J\overline{\nabla}_{X}X
\end{eqnarray*}
Using (\ref{a6}) , we have
\begin{eqnarray*}
\overline{\nabla}_{X}TX + \overline{\nabla}_{X}NX = J\overline{\nabla}_{X}X
\end{eqnarray*}
Using (\ref{a5n1}) and (\ref{a5n2}), we obtain
\begin{eqnarray*}
\nabla_{X}TX + g(X,TX)\mathcal{H} - Xg(\mathcal{H},NX) + \nabla_{X}^{\perp}NX = T\nabla_{X}X + N\nabla_{X}X + g(X,X)J\mathcal{H}
\end{eqnarray*}
Taking the inner product of above equation with $J\mathcal{H}$, we get
\begin{eqnarray*}
g(\nabla_{X}^{\perp}NX,J\mathcal{H}) = g(N\nabla_{X}X, J\mathcal{H}) + g(X,X)g(J\mathcal{H},J\mathcal{H})
\end{eqnarray*}
or
\begin{eqnarray}\label{uu08}
g(\nabla_{X}^{\perp}NX,J\mathcal{H}) =  g(X,X)\|\mathcal{H}\|^{2}
\end{eqnarray}
Also for any $X \in TM$, we have
\begin{equation*}\label{uu8}
(\overline{\nabla}_{X}J)\mathcal{H} = \overline{\nabla}_{X} J\mathcal{H} - J\overline{\nabla}_{X}\mathcal{H}
\end{equation*}
using (\ref{a5n2}), we have
\begin{equation}\label{uu9}
(\overline{\nabla}_{X}J)\mathcal{H} = -Xg(\mathcal{H},J\mathcal{H}) + \nabla_{X}^{\perp}J\mathcal{H} -JXg(\mathcal{H},\mathcal{H}) + J\nabla_{X}^{\perp}\mathcal{H}
\end{equation}
Let $\mathcal{P}X$ and $\mathcal{N}X$ denotes the tangential and normal components of $(\overline{\nabla}_{X}J)\mathcal{H}$. Now taking the inner product of (\ref{uu9}), with $NX$, we get
\begin{equation*}\label{uu11}
 g(\mathcal{N}X, NX) = g(\nabla_{X}^{\perp}J\mathcal{H}, NX) -sin^{2}\theta \|\mathcal{H}\|^{2}g(X,X)
\end{equation*}
or
\begin{equation}\label{uu12}
 g(\nabla_{X}^{\perp}NX, J\mathcal{H}) = - sin^{2}\theta \|\mathcal{H}\|^{2}g(X,X) -  g(\mathcal{N}X, NX)
\end{equation}
 From (\ref{uu08}) and (\ref{uu12}), we have
\begin{eqnarray}\label{uu13}
 cos^{2}\theta \|\mathcal{H}\|^{2}g(X,X) +  g(\mathcal{N}X, NX) = 0
\end{eqnarray}
 from (\ref{a5n4}), we have
\begin{eqnarray}\label{uu14}
 g(\mathcal{N}X, NX) = sin^{2}\theta g(X,X)\|\mathcal{H}\|^{2}
\end{eqnarray}
Thus from (\ref{a5n4}),(\ref{uu13}) and  (\ref{uu14}) we have
\begin{eqnarray}\label{uu14}
g(X,X)\|\mathcal{H}\|^{2} = 0
\end{eqnarray}
Since $\mathcal{N}$ is totally umbilical proper hemi-slant submanifolds. Hence from (\ref{uu14}), we conclude that  $\mathcal{N}$ is totally geodesic in $\overline{\mathcal{N}}$.
\end{proof}
\section{Cohomology of Hemi-slant submanifolds of nearly Kaehler manifold}
In this section we study the de Rham cohomology class for a closed hemi-slant submanifold in a nearly Kaehler manifold.

Let $\{e_{1}, ..., e_{l}, sec\theta Te_{1}, ..., sec\theta Te_{l}, E_{1}, ..., E_{m}, JE_{1}, ..., JE_{m}\}$ be the local orthonormal frame on $\mathcal{N}$ and $\{\omega^{1}, ..., \omega^{m}, \omega^{1^{*}}, ..., \omega^{l^{*}}, \varpi^{1}, ..., \varpi^{m^{*}}\}$ be the dual frame of one forms, then for any $Z \in D^{\perp}$, we have
\begin{equation*}\label{ch1}
\omega^{i}(Z) = 0,\hspace{.3cm} \omega^{i}(e_{j})= \delta^{i}_{j}\hspace{.3cm} \text{for}\hspace{.3cm} i,j = 1, 2, ..., 2l
\end{equation*}
and
\begin{equation*}\label{ch2}
\omega^{i^{*}}(Je_{j}) =  \delta^{i}_{j}, i^{*} = l+j, ..., 2l, j= 1, ...,m
\end{equation*}
where $e_{l+j} = sec\theta Te_{j}$. Then
\begin{equation}\label{ch3}
\omega = \omega^{1}\wedge ... \wedge \omega^{2l}
\end{equation}
then $\omega$ defines the $2l$-form on $\mathcal{N}$ and is globally defined on $\mathcal{N}$.

Now we are able to prove the main result of this section.

\begin{theorem}
Let $\mathcal{N}$ be a closed hemi-slant submanifold of a nearly Kaehler manifold $\mathcal{\overline{N}}$. If the normal bundle is parallel and 
\begin{equation*}\label{ch4}
g(A_{JW}X+ sec^{2}\theta A_{JW}X, TX) = 0,
\end{equation*}
for $X \in D^{\theta}$ and $Z \in D^{\perp}$
then the $2l$-form $\omega$ defines a de Rham cohomology class $[\omega]$ in $\mathbf{H}^{2l}(M,\mathbb{R})$. Now, if $D^{\perp}$ is minimal and $D^{\theta}$ is integrable then the cohomology class is non-trivial.
\end{theorem}

\begin{proof}
From (\ref{ch3}), we have
\begin{equation}\label{ch5}
d\omega = \sum_{i=1}^{2l}(-1)^{i}\omega^{1}\wedge ...\wedge ... \wedge \omega^{2l}.
\end{equation}
suppose $\omega$ defines a cohomology class $[\omega] \in \mathbf{H}^{2l}(M,\mathbb{R})$. Now $d\omega = 0$, if and only if
\begin{equation}\label{ch6}
d\omega(Z,W, X_{1}, ..., X_{2l-1}) = 0
\end{equation}
and
\begin{equation}\label{ch7}
d\omega(Z, X_{1}, ..., X_{2l}) = 0
\end{equation}
for $Z, W \in D^{\perp}$ and $X_{1}, ...,X_{2l} \in D^{\theta}$. Thus (\ref{ch6}) holds if and only if $[Z,W] \in D^{\perp}$, which is the condition of integrability of $D^{\perp}$, which holds because of the  theorem 3.5. The equation (\ref{ch7}) satisfies if and only if $D^{\theta}$ is minimal distribution, which holds due to the lemma 3.7. Hence $\omega$ is a closed form on $\mathcal{N}$ and defines a cohomology class given by  $[\omega] \in \mathbf{H}^{2l}(M,\mathbb{R})$.

Now suppose $D^{\theta}$ is integrable and $D^{\perp}$ is minimal. By using similar arguments for the frame and dual frame of $D^{\perp}$, we can show that $\varpi$ is closed. since $\omega$ and $\varpi$ are closed, we get $\delta\omega=0$, which shows that $\mathcal{N}$ is harmonic $2l$-form. As $D\neq0$ shows that $\omega\neq o$, the cohomology class is non-trivial in  $\mathbf{H}^{2l}(M,\mathbb{R})$. which completes the proof.

\end{proof}

{\bf{References}}

 \end{document}